\documentclass[twoside,leqno,10pt, A4]{amsart}
\usepackage{amsfonts}
\usepackage{amsmath}
\usepackage{amscd}
\usepackage{amssymb}
\usepackage{amsthm}
\usepackage{amsrefs}
\usepackage{latexsym}
\usepackage{mathrsfs}
\usepackage{bbm}
\usepackage{enumerate}
\usepackage{graphicx}

\usepackage{amsfonts}
\usepackage{amsmath}
\usepackage{amscd}
\usepackage{amssymb}
\usepackage{amsthm}
\usepackage{amsrefs}
\usepackage{latexsym}
\usepackage{mathrsfs}
\usepackage{bbm}
\usepackage{amscd}
\usepackage{amssymb}
\usepackage{amsthm}
\usepackage{amsrefs}
\usepackage{latexsym}
\usepackage{mathrsfs}
\usepackage{bbm}
\usepackage{enumerate}
\usepackage{graphicx}
\usepackage{color}
\begin{document}

\newtheorem{theorem}[subsection]{Theorem}
\newtheorem{proposition}[subsection]{Proposition}
\newtheorem{lemma}[subsection]{Lemma}
\newtheorem{corollary}[subsection]{Corollary}
\newtheorem{conjecture}[subsection]{Conjecture}
\newtheorem{prop}[subsection]{Proposition}
\numberwithin{equation}{section}
\newcommand{\mr}{\ensuremath{\mathbb R}}
\newcommand{\mc}{\ensuremath{\mathbb C}}
\newcommand{\dif}{\mathrm{d}}
\newcommand{\intz}{\mathbb{Z}}
\newcommand{\ratq}{\mathbb{Q}}
\newcommand{\natn}{\mathbb{N}}
\newcommand{\comc}{\mathbb{C}}
\newcommand{\rear}{\mathbb{R}}
\newcommand{\prip}{\mathbb{P}}
\newcommand{\uph}{\mathbb{H}}
\newcommand{\fief}{\mathbb{F}}
\newcommand{\majorarc}{\mathfrak{M}}
\newcommand{\minorarc}{\mathfrak{m}}
\newcommand{\sings}{\mathfrak{S}}
\newcommand{\fA}{\ensuremath{\mathfrak A}}
\newcommand{\mn}{\ensuremath{\mathbb N}}
\newcommand{\mq}{\ensuremath{\mathbb Q}}
\newcommand{\half}{\tfrac{1}{2}}
\newcommand{\f}{f\times \chi}
\newcommand{\summ}{\mathop{{\sum}^{\star}}}
\newcommand{\chiq}{\chi \bmod q}
\newcommand{\chidb}{\chi \bmod db}
\newcommand{\chid}{\chi \bmod d}
\newcommand{\sym}{\text{sym}^2}
\newcommand{\hhalf}{\tfrac{1}{2}}
\newcommand{\sumstar}{\sideset{}{^*}\sum}
\newcommand{\sumprime}{\sideset{}{'}\sum}
\newcommand{\sumprimeprime}{\sideset{}{''}\sum}
\newcommand{\sumflat}{\sideset{}{^\flat}\sum}
\newcommand{\shortmod}{\ensuremath{\negthickspace \negthickspace \negthickspace \pmod}}
\newcommand{\V}{V\left(\frac{nm}{q^2}\right)}
\newcommand{\sumi}{\mathop{{\sum}^{\dagger}}}
\newcommand{\mz}{\ensuremath{\mathbb Z}}
\newcommand{\leg}[2]{\left(\frac{#1}{#2}\right)}
\newcommand{\muK}{\mu_{\omega}}
\newcommand{\thalf}{\tfrac12}
\newcommand{\lp}{\left(}
\newcommand{\rp}{\right)}
\newcommand{\Lam}{\Lambda_{[i]}}
\newcommand{\lam}{\lambda}
\def\L{\fracwithdelims}
\def\om{\omega}
\def\pbar{\overline{\psi}}
\def\phis{\phi^*}
\def\lam{\lambda}
\def\lbar{\overline{\lambda}}
\newcommand\Sum{\Cal S}
\def\Lam{\Lambda}
\newcommand{\sumtt}{\underset{(d,2)=1}{{\sum}^*}}
\newcommand{\sumt}{\underset{(d,2)=1}{\sum \nolimits^{*}} \widetilde w\left( \frac dX \right) }

\newcommand{\hf}{\tfrac{1}{2}}
\newcommand{\af}{\mathfrak{a}}
\newcommand{\Wf}{\mathcal{W}}

\newtheorem{mylemma}{Lemma}
\newcommand{\intR}{\int_{-\infty}^{\infty}}

\theoremstyle{plain}
\newtheorem{conj}{Conjecture}
\newtheorem{remark}[subsection]{Remark}

\makeatletter
\def\widebreve{\mathpalette\wide@breve}
\def\wide@breve#1#2{\sbox\z@{$#1#2$}%
     \mathop{\vbox{\m@th\ialign{##\crcr
\kern0.08em\brevefill#1{0.8\wd\z@}\crcr\noalign{\nointerlineskip}%
                    $\hss#1#2\hss$\crcr}}}\limits}
\def\brevefill#1#2{$\m@th\sbox\tw@{$#1($}%
  \hss\resizebox{#2}{\wd\tw@}{\rotatebox[origin=c]{90}{\upshape(}}\hss$}
\makeatletter

\title[Non-vanishing of quadratic twists of modular $L$-functions of prime-related moduli]{Non-vanishing of quadratic twists of modular $L$-functions of prime-related moduli}

\author{Peng Gao and Liangyi Zhao}

\begin{abstract}
 In this paper, we study central values of the family of quadratic twists of modular $L$-functions of moduli $8p$, with $p$ ranging over odd primes.  Assuming the truth of the generalized Riemann hypothesis, we establish a positive proportion non-vanishing result for the corresponding $L$-values.
\end{abstract}

\maketitle



\section{Introduction}
\label{sec 1}

The non-vanishing of the central values of $L$-functions is very important subject in number theory, as these central values carry deep arithmetic information.  S. Chowla \cite{chow} conjectured that an $L$-functions attached to a primitive Dirichlet character does not have a zero at $s=\tfrac{1}{2}$.  There are now two major approaches towards establishing positive proportion non-vanishing results for families of $L$-function. One of them is the mollifier method that allows one to achieve a non-vanishing result unconditionally via evaluating mollified moments. For example, K. Soundararajan \cite{sound1} applied this method to derive the well-known result that at least $87.5\%$ of the members of the family of quadratic Dirichlet $L$-functions do not vanish at the central point. \newline

The other {\it modus operandi} is provided by the density conjecture of N. Katz and P. Sarnak \cites{KS1, K&S}. In this approach, one computes the one-level density of low-lying zeros of families of $L$-functions to derive a corresponding non-vanishing result. In order to obtain better results, one often needs to make use of the generalized Riemann hypothesis (GRH).  For example, A. E. \"{O}zluk and C. Snyder  \cite{O&S} used this process to show that at least $15/16=93.75\%$ of the members of the family of quadratic Dirichlet $L$-functions have non-vanishing central values. Optimizing the test function as  in \cites{B&F, ILS}, this percentage can be improved to $(19-\cot 1/4)/16$, approximately 94.27\%. \newline

 Both of the above two ways may fail to yield any positive proportion non-vanishing result.  In the first instance, the mollifier method requires one to evaluate at least two mollified moments of the family of $L$-functions to deduce the desired result and evaluating
 higher mollified moments may be challenging.  In the second approach, one needs the one-level density result to hold for test functions whose Fourier transforms have long supports in order to deduce a positive proportion non-vanishing result. The current technology often fall short of this, even under GRH. \newline

  To circumvent the above difficulties, one may apply the mollifier method with more flexibility by observing that the method works if one can bound one mollified moment from below and another from above instead of evaluating both moments asymptotically.  Often, the lower bounds are relatively easier to achieve, and one may applies various tools obtain the upper bounds. For example, S. Baluyot and K. Pratt \cite{B&P} were able to applied sieve method to bound the second mollified moment of the quadratic family of Dirichlet $L$-functions of prime moduli from above to show that more than nine percent of the members of do not vanish at the central value.  Another powerful method for establishing upper bounds was developed by K. Soundararajan in \cite{Sound2009} with a refinement by A. J. Harper \cite{Harper}, which allows one to bound the $L$-function by a very short Dirichlet polynomial over the primes under GRH. \newline

 One may also consider different choices of mollifiers, even though the optimal mollifiers are essentially found in many cases (see for example \cite{sound1}).  Applying these alternative mollifiers may bypass the difficulty of evaluating the mollified moments using the optimal ones. One systematic way for constructing new mollifiers originates from the study on sharp bounds for the moments of families of $L$-functions in the work of M. Radziwi{\l\l}  and K. Soundararajan \cite{Radziwill&Sound} and the work of W. Heap and K. Soundararajan \cite{H&Sound}. Compared to the optimal mollifiers used in the study of the non-vanishing issues, these new mollifiers can be thought of as obtained from using the Euler product to approximate the inverses of an $L$-function, while the optimal mollifiers such as the one given in \cite{sound1} are obtained from the consideration of employing Dirichlet series to approximate such an inverse. Thus, though suboptimal, these new mollifiers have the advantages of retaining essentially an Euler product structure, hence much easier to deploy when incorporated in the evaluations of the mollified moments. \newline

Although the new mollifiers are primarily employed to obtain sharp bounds for the moments of families of $L$-functions, there are emerging cases for applying them to study various subjects in number theory. For example, they have been used by S. Lester and M. Radziwi{\l\l} \cite{LR21} to study sign changes of Fourier coefficients of half-integral weight modular
forms, by C. David, A. Florea and M. Lalin \cite{DFL21} to establish a positive proportion non-vanishing result of cubic $L$-functions over function field, by  M. Radziwi{\l\l}  and K. Soundararajan \cite{Radziwill&Sound} as well as by H. M. Bui, N. Evans, S. Lester and K. Pratt \cite{BELP} to establish central limit theorem for the central values of $L$-functions, by P. Gao and L. Zhao \cite{G&Zhao10} to establish a positive proportion non-vanishing result of cubic Dirichlet $L$-functions. We point out here the main result of S. Lester and M. Radziwi{\l\l} \cite{LR21} essentially relies on evaluations of the first and second mollified moments of quadratic modular $L$-functions and in particular implies  a positive proportion non-vanishing result of the corresponding $L$-values. \newline

 In view of the above applications of the Euler product-like mollifiers, it is now clear that one may be able to apply them to obtain positive proportion of non-vanishing results on central values of $L$-functions in general, especially when using the optimal mollifiers is beyond the reach. It is the aim of the paper to illustrate this approach by addressing the non-vanishing issue of quadratic twists of modular $L$-functions at the central point. To state our result, we fix a holomorphic Hecke eigenform $f$ of weight $\kappa \equiv 0 \pmod 4$ for the full modular group $SL_2 (\mathbb{Z})$. We write $e(z) $ for $e^{2 \pi i z}$ so that the Fourier expansion of $f$ at infinity is
\[
f(z) = \sum_{n=1}^{\infty} \lambda_f (n) n^{(\kappa -1)/2} e(nz).
\]
 Here $\lambda_f (n) \in \mr$ witht $\lambda_f (1) =1$ and $|\lambda_f(n)|
\leq d(n)$, for all $n \geq 1$ by Deligne's bound \cite{D}, where $d(n)$ is the divisor function of $n$.  Moreover, let $\chi_d=\leg {d}{\cdot} $ denote the Kronecker symbol.  Hence the twisted modular $L$-function $L(s, f \otimes \chi_d)$ for $\Re(s) > 1$ is defined as
\begin{align*}
L(s, f \otimes \chi_d) &= \sum_{n=1}^{\infty} \frac{\lambda_f(n)\chi_d(n)}{n^s}
 = \prod_{p\nmid d} \left(1 - \frac{\lambda_f (p) \chi_d(p)}{p^s}  + \frac{1}{p^{2s}}\right)^{-1}.
\end{align*}

 Define $\epsilon(d) =1$ for $d>0$ and  $\epsilon(d) =-1$ for $d<0$. Then the function $L(s, f \otimes \chi_d)$ satisfies the functional equation given by
\begin{align*}
\Lambda (s, f \otimes \chi_d) = \left(\frac{|d|}{2\pi} \right)^s \Gamma (s + \tfrac{\kappa -1}{2}) L(s, f \otimes \chi_d)
= i^\kappa \epsilon(d ) \Lambda (1- s, f \otimes \chi_d).
\end{align*}

  Let $X$ be a large number and we reserve the letter $p$ for primes throughout the paper. We consider the family $\{ L(s, f \otimes \chi_{8p}) : 2 < p < X\}$. Note that our choice for $\kappa$ ensures that the corresponding central $L$-values do not automatically vanish due to the sign of the functional equation.  Our main result in this paper establishes that a positive proportion of the members of the above family of $L$-functions do not vanish at the central point.
\begin{theorem}
\label{thm:nonvanish}
Assume the truth of GRH. With the notations above, there exist infinitely many odd primes $p$  such that $L(\tfrac{1}{2},f \otimes \chi_{8p}) \neq 0$.  More precisely, we have
\begin{align}
\label{Nest}
\sum_{\substack{2< p \leq X  \\ L(\frac{1}{2},f \otimes\chi_{8p})\neq 0}}(\log p) &\gg  \sum_{\substack{2< p \leq X  }}(\log p) .
\end{align}
\end{theorem}

\section{Preliminaries}
\label{sec 2}

\subsection{Various Sums}
\label{sec2.4}

   We gather a few results concerning sums over primes.
\begin{lemma}
\label{RS} Let $x \geq 2$. We have, for some constants $b_1, b_2$,
\begin{align}
\label{merten}
\sum_{p\le x} \frac{1}{p} =& \log \log x + b_1+ O\Big(\frac{1}{\log x}\Big), \quad \mbox{and} \\
\label{M1}
\sum_{p\le x} \frac{\lambda^2_f(p)}{p} =& \log \log x + b_2+ O\Big(\frac{1}{\log x}\Big).
\end{align}
 Also, for any integer $j \geq 1$, we have
\begin{equation}
\label{mertenpartialsummation}
\sum_{p\le x} \frac {(\log p)^j}{p} = \frac {(\log x)^j}{j} + O((\log x)^{j-1}).
\end{equation}
 Let $\chi$ be a primitive Dirichlet character modulo $q$ and assume that GRH holds for $L(s, \chi)$, we have
\begin{align}
\label{PIT}
 \sum_{p \leq x }\log p \cdot \chi(p) =\delta_{\chi=\chi_0}x+O(\sqrt{x} \left(\log 2qx)^2 \right),
\end{align}
 where we define $\delta_{\chi=\chi_0}=1$ if $\chi=\chi_0$ and $\delta_{\chi=\chi_0}=0$ otherwise.
\end{lemma}
\begin{proof}
The expression in \eqref{merten} is a well-known formula due to Merten (see \cite[Theorem 2.7]{MVa1}) and \eqref{M1} follows from the Rankin-Selberg theory for $L(s, f)$ (see \cite[Chapter 5]{iwakow}).  The formula in \eqref{mertenpartialsummation} follows from \eqref{merten} by partial summation.  Finally,  \eqref{PIT} is a special case of \cite[Theorem 5.15]{iwakow}.
\end{proof}

  For any positive odd integer $c$, we write $\psi_{c}$ for the Dirichlet character with $\psi_{c}(n) = \leg {n}{c} $ for $n \in \mz$. We also define $\delta_{c=\square}=1$ if $c$ is a perfect square and $\delta_{c=\square}=0$ otherwise. Let $\Phi$ be a smooth, non-negative function compactly supported on $[1/2,5/2]$ satisfying $\Phi(x) =1$ for $x\in [1,2]$. We denote the Mellin transform of $\Phi(x)$ by ${\widehat \Phi}(s)$, that is,
\begin{equation*}
{\widehat \Phi}(s) = \int\limits_{0}^{\infty} \Phi(x)x^{s}\frac {\dif x}{x}.
\end{equation*}
   We note the following result concerning the smoothed quadratic character sums given in \cite[Lemma 2.8]{G&Zhao11}.
\begin{lemma}
\label{lemma logd}
Assume the truth of GRH. Let $c$ be a positive odd integer and $\Phi(X)$ be a smooth function. Then for any $\varepsilon>0$,
\begin{equation*}
 \sum_{(p,2)=1} (\log p) \chi_{8p}(c) \Phi \left( \frac {p}X \right) =  \delta_{c=\square}\widehat{\Phi}(1)X+O \left( X^{1/2+\varepsilon}\log
 \log (c+2) \right).
\end{equation*}
\end{lemma}

\subsection{Twisted First Moment}
\label{smoothsum}

   We recall the following approximate functional equation concerning $L(\tfrac{1}{2}, f \otimes \chi_{8p})$ given in \cite[p. 451]{sound1}.
\begin{lemma}[Approximate functional equation]
\label{lem:AFE}
  For any odd prime $p$, we have
\begin{align*}
\begin{split}
 L(\tfrac{1}{2}, f \otimes \chi_{8p}) = & 2\sum^{\infty}_{\substack{n=1}} \frac{\lambda_f(n)\chi_{8p}(n)}{\sqrt{n}} V
\left(\frac{ n}{p} \right),
\end{split}
\end{align*}
 where for any real number $t>0$,
\begin{align}
\label{eq:Vdef}
 V(t) = \frac{1}{2 \pi i} \int\limits\limits_{(2)}  \left(\frac {\pi}{8} \right)^{-s}
\left ( \frac {\Gamma(\kappa/2+s)}{\Gamma(\kappa/2)} \right ) t^{-s} \frac {\dif s}{s}.
\end{align}
\end{lemma}

In the course of proving Theorem~\ref{thm:nonvanish} , we need the following result on the twisted first moment of quadratic twists of $L(\tfrac{1}{2}, f \otimes \chi_{8p})$.
\begin{proposition}
\label{twistedfirstmoment}
  Assume the truth of GRH and use the same notations as above.  Let $X$ be a large real number and $\ell$ a fixed positive odd integer.  Further write $\ell$ uniquely as $\ell=\ell_1\ell^2_2$ with $\ell_1$ square-free.  We have
\begin{align} \label{eq:1}
 \sum_{(p,2)=1}(\log p)L(\tfrac{1}{2}, f \otimes \chi_{8p}) \chi_{8p}(\ell) \Phi\leg{p}{X} =C \widehat{\Phi}(1)\frac {\lambda_f(l_1)}{\sqrt{l_1}g(l)}X +O \left(X^{\half+\varepsilon}\ell_1^{\varepsilon}\right ),
\end{align}
  where $C$ is an explicit constant and $g(l)$ is a multiplicative function such that $1/g(p)=1+O(1/p)$.
\end{proposition}

\begin{proof}
The approximate functional equation in Lemma \ref{lem:AFE} gives
\begin{equation*} \label{splitting}
\mathcal{M} :=\sum_{(p,2)=1} (\log p)L(\tfrac{1}{2}, f \otimes \chi_{8p}) \chi_{8p}(\ell) \Phi\leg{p}{X}
=2
\sum^{\infty}_{m=1} \frac {\lambda_f(m)}{m^{1/2}}\sum_{(p,2)=1}(\log p)\chi_{8p}(m\ell )V(\frac {m}{p})  \Phi\left(\frac{p}{X}\right).
\end{equation*}

   Due to the rapid decay of $\Phi$, we arrive at
\begin{align*}
 \sum_{(p,2)=1}(\log p)\chi_{8p}(m\ell )V \left( \frac {m}{p} \right) \Phi\left(\frac{p}{X}\right) =& \sum^{\infty}_{n=1}  \chi_{8n}(m\ell )
 \Lambda(n) V \left( \frac {m}{n} \right) \Phi \left( \frac {n}X \right)
 +O \left(\sum_{\substack{p^j \leq X^{1+\varepsilon}, j \geq 2 }} (\log p)\Phi \left( \frac {p^j}X \right)   \right ).
\end{align*}

  By \eqref{PIT}, the $O$-term above is $\ll X^{1/2+\varepsilon}$. To deal with the first term on the right-hand side above, we apply Mellin
   inversion to obtain that
\begin{align}
\label{twistedcharsum}
\sum^{\infty}_{n=1}  \chi_{8n}(m\ell ) \Lambda(n) V(\frac {m}{n}) \Phi \left( \frac {n}X \right)
=& -\frac {\chi_8(m\ell )}{2\pi i}\int\limits_{(2)} \frac {L'(s, \psi_{m\ell})}{L(s, \psi_{m\ell})} \widehat{f}(s)X^s \dif s,
\end{align}
  where
\begin{equation*}
\widehat{f}(s) = \int\limits_0^{\infty} V \left(\frac{m}{Xx} \right) \Phi(x) x^{s-1} \dif x.
\end{equation*}

   Now repeated integration by parts yields the bound
\begin{equation*}
 \widehat{f}(s) \ll (1 + |s|)^{-E} \left( 1 + m/X \right)^{-E},
\end{equation*}
for any $\Re(s) >0$ and any integer $E>0$. \newline

   We evaluate the integral in \eqref{twistedcharsum} by shifting the line of integration to $\Re(s)=1/2+\varepsilon$ passing a pole at
   $s=1$ with residue $-\widehat{f}(1)X$ only if $m\ell$ is a perfect square.  The integration on $\Re(s)=1/2+\varepsilon$ can be estimated
   to be $O(X^{1/2+\varepsilon}\log \log (m\ell+2))$, thanks to the rapid decay of $\widehat{f}$ on the vertical line and the estimate (see
  \cite[Theorem 5.17]{iwakow}) that under GRH, for $\Re(s) \geq 1/2+\varepsilon$,
\begin{align*}
  \frac {L'(s, \psi_c)}{L(s, \psi_c)}  \ll \log\log \big ((c+2)(1+|s|)\big).
\end{align*}

  To treat the contribution from the poles, we deduce via the expression for $V$ in \eqref{eq:Vdef} that
\begin{align*}
 \widehat{f}(1) = \int\limits_0^{\infty} V\left(\frac{m}{Xx}  \right) \Phi(x) \dif x = \frac{1}{2 \pi i} \int\limits_{(2)}
 \leg{X}{m}^s\widehat{\Phi}(1+s)   \left(\frac{8}{\pi}\right)^{s} \frac {\Gamma(\kappa/2+s)}{\Gamma(\kappa/2)}
 \frac {\dif s}{s}.
\end{align*}	

Note that $m\ell$ is a perfect square if and only if $m$ is a square multiple of $\ell_1$.  So we may replace $m$ by $\ell_1 m^2$ and the contribution from the poles to $\mathcal{M}$ is
\begin{align}
\label{M0integral}
  \frac {2X}{\sqrt{\ell_1}}\frac{1}{2 \pi i} \int\limits_{(2)} \widehat{\Phi} \left( 1+s \right)   \left(\frac{8}{\pi}\right)^{s} \frac
  {\Gamma(\kappa/2+s)}{\Gamma(\kappa/2)}\leg{X}{\ell_1}^s \Big (\sum_{\substack{m \geq 1 \\ (m,2)=1}}\frac {\lambda_f(l_1m^2)}{m^{1+2s}}\Big )\frac {\dif s}{s}.
\end{align}

  To evaluate the last sum above, we recall that the symmetric square $L$-function $L(s, \operatorname{sym}^2 f)$ of $f$ is defined for $\Re(s)>1$ by (see \cite[(25.73)]{iwakow})
\begin{align*}
 L(s, \operatorname{sym}^2 f)=& \zeta(2s) \sum_{n \geq 1}\frac {\lambda(n^2)}{n^s}=\prod_{p} \left( 1-\frac {\lambda(p^2)}{p^s}+\frac {\lambda(p^2)}{p^{2s}}-\frac {1}{p^{3s}} \right)^{-1}.
\end{align*}
  It follows from a result of G. Shimura \cite{Shimura} that the corresponding completed $L$-function
\begin{align*}
 \Lambda(s, \operatorname{sym}^2 f)=& \pi^{-3s/2}\Gamma \left( \frac {s+1}{2} \right)\Gamma \left( \frac {s+\kappa-1}{2} \right) \Gamma \left( \frac {s+\kappa}{2} \right) L(s, \operatorname{sym}^2 f).
\end{align*}
  is entire and satisfies the functional equation
$\Lambda(s, \operatorname{sym}^2 f)=\Lambda(1-s, \operatorname{sym}^2 f)$.
 Combining this with \cite[(5.8)]{iwakow} and apply the convexity bounds (see \cite[Exercise 3, p.
  100]{iwakow}) for $L$-functions, we deduce that
\begin{align}
\label{Lsymest}
\begin{split}
  L(s, \operatorname{sym}^2 f) \ll & \left( 1+|s| \right)^{(3(1-\Re(s)))/2+\varepsilon}, \quad 0 \leq \Re(s) \leq 1.
\end{split}
\end{align}

  We apply the above notations to write
\begin{align*}
 \sum_{\substack{m \geq 1 \\ (m,2)=1}}\frac {\lambda_f(l_1m^2)}{m^{1+2s}}= \lambda_f(l_1)\zeta(2+4s)^{-1}L(1+2s, \operatorname{sym}^2 f)  \prod_{p | 2l} \left( 1+O \left( \frac {1}{p^{1+2s}} \right) \right)^{-1}
\end{align*}

  We then evaluate the integral in \eqref{M0integral} by moving the contour of integration to $-1/2 + \varepsilon$, crossing a simple pole at $s=0$ only. The integral on the new line is estimated, using \eqref{Lsymest}, to be
\begin{equation*}
  \ll X^{1/2+\varepsilon}|l|^{\varepsilon}.
\end{equation*}

  The residue of the pole in the above process can be easily computed and this leads to the main term on the right-hand side of \eqref{eq:1}.  This completes the proof of the proposition.
\end{proof}

\subsection{Upper bound for $\log |L(\tfrac{1}{2}, f \otimes \chi_{8d})|$ }

  Although we are primarily interested in $L(\tfrac{1}{2}, f \otimes \chi_{8p})$ in this paper, we include here an upper bound result from \cite{S&Y} on $\log |L(\tfrac{1}{2}, f \otimes \chi_{8d})|$ for any odd, square-free integer $d>0$ in terms of a short Dirichlet polynomial over the primes.
\begin{lemma}
\label{lem: logLbound}
 Assume the truth of GRH for $\zeta(s)$ and for $L(s, \chi_{8d})$ for a fixed odd, square-free integer $d>0$. We have for $d \leq X$, $2 \leq x \leq X$,
\begin{align*}
\log |L( \tfrac{1}{2}, f \otimes \chi_{8d})|   \leq  \sum_{\substack{  p \leq x }} \frac{\chi_{8d} (p)\lambda_f(p)}{p^{1/2+1/\log x}}
 \frac{\log (x/p)}{\log x} -\frac 12\log \log x-\sum_{\substack{  p | d  }} \frac{\lambda^2_f(p)-2}{2p}
 +\frac{2\log X}{\log x} + O\left( \frac {\log X}{x^{1/2}\log x}+1 \right).
\end{align*}
\end{lemma}
\begin{proof}
  We write $\lambda_f(p)=\alpha_p+\beta_p$ where $\alpha_p\beta_p=1$ and note that $|\alpha_p|=|\beta_p|=1$ by Deligne's theorem \cite{D}. Let $\lambda_0=0.4912\ldots$ be the unique positive real number satisfying $e^{-\lambda_0} = \lambda_0+ \lam^2_0/2$.  We set $z_1=z_2=0$ in \cite[(6.1)]{S&Y}.  Note that $\alpha^l_p+\beta^l_p \in \mr$ for any integer $l \geq 1$.  This leads to, for $\lambda \geq \lambda_0$,
\begin{align}
\label{logLupperbound}
\log |L(\tfrac{1}{2}, f \otimes \chi_{8d})|   \leq \sum_{\substack{  p^l \leq x \\ l \geq 1 }} \frac{\chi_{8d} (p^l)(\alpha^l_p+\beta^l_p) }{lp^{l(1/2+\lambda/\log x)}}
 \frac{\log (x/p^l)}{\log x} +(1+\lambda)\frac{\log X}{\log x} + O(1).
\end{align}
  We set $\lambda=1$ and note that the terms with $l \geq 3$ are $O(1)$ (see the paragraph below \cite[(6.1)]{S&Y}). Moreover,
  \[ \alpha_p+\beta_p=\lambda_f(p) \quad \mbox{and} \quad  \alpha^2_p+\beta^2_p=\lambda^2_f(p)-2=\lambda_f(p^2)-1. \] 
Hence
\begin{align}
\begin{split}
\label{logLupperbound1}
 \log |L(\tfrac{1}{2}, f \otimes \chi_{8d})| 
  \leq& \sum_{\substack{  p \leq x }} \frac{\chi_{8d} (p)\lambda_f(p)}{p^{1/2+1/\log x}}
 \frac{\log (x/p)}{\log x} +  \frac 12
 \sum_{\substack{  p \leq  x^{1/2} }} \frac{\chi_{8d} (p^2)(\lambda^2_f(p)-2)}{p^{1+2/\log x}}  \frac{\log (x/p^2)}{\log x}
 +\frac{2\log X}{\log x} + O(1) \\
 =&  \sum_{\substack{  p \leq x }} \frac{\chi_{8d} (p)\lambda_f(p)}{p^{1/2+1/\log x}}
 \frac{\log (x/p)}{\log x} + \frac 12
 \sum_{\substack{  p \leq  x^{1/2} \\ p \nmid d }} \frac{(\lambda^2_f(p)-2)}{p^{1+2/\log x}}  \frac{\log (x/p^2)}{\log x}
 +\frac{2\log X}{\log x}+ O(1).
\end{split}
\end{align}

   Now
\begin{align}
\begin{split}
\label{sumoverpeval}
 \sum_{\substack{  p \leq  x^{1/2} \\ p \nmid d  }} \frac{(\lambda^2_f(p)-2)}{p^{1+2/\log x}}  \frac{\log (x/p^2)}{\log x}
 =\sum_{\substack{  p \leq  x^{1/2} \\ p \nmid d  }} \frac{\lambda^2_f(p)-2}{p^{1+2/\log x}}
 - 2\sum_{\substack{  p \leq  x^{1/2}  \\ p \nmid d  }} \frac{(\lambda^2_f(p)-2)}{p^{1+2/\log x}}  \frac{\log p}{\log x}.
\end{split}
\end{align}

  Applying Deligne's bound, we see that
\begin{align*}
\begin{split}
|\lambda^2_f(p)-2| \leq |\lambda^2_f(p)|+2 \leq d(p^2)+2 \ll 1.
\end{split}
\end{align*}
 It follows from this and and \eqref{mertenpartialsummation} in Lemma \ref{RS} that
\begin{align}
\label{errorsumlogp}
  \sum_{\substack{ p \leq x^{1/2}  \\ p \nmid d  }}  \frac{(\lambda^2_f(p)-2)}{p^{1+2/\log x}}  \frac{\log p}{\log x} \ll  \frac 1{\log x} \sum_{\substack{ p \leq x^{1/2} }}  \frac{\log p}{p} \ll 1.
\end{align}

For the other term on the right-hand side of \eqref{sumoverpeval},
\begin{align}
\begin{split}
 \sum_{\substack{  p \leq  x^{1/2} \\ p \nmid d  }} \frac{\lambda^2_f(p)-2}{p^{1+2/\log x}}
 =& \sum_{\substack{  p \leq  x^{1/2} \\ p \nmid d  }} \frac{\lambda^2_f(p)-2}{p} +\sum_{\substack{  p \leq  x^{1/2}  \\ p \nmid d  }} (\lambda^2_f(p)-2)\ \Big (\frac{1}{p^{1+2/\log x}}-\frac 1p \Big ) \\
 =& \sum_{\substack{  p \leq  x^{1/2} \\ p \nmid d  }} \frac{\lambda^2_f(p)-2}{p} +O\Big ( \sum_{\substack{  p \leq  x^{1/2} }}  \frac{\log p}{p \log x}\Big ).
\end{split}
\end{align}
Now from \eqref{errorsumlogp}, the last $O$-term above is $\ll 1$. \newline

  Now, we have
\begin{align}
\begin{split}
 \sum_{\substack{  p \leq  x^{1/2} \\ p \nmid d  }} \frac{\lambda^2_f(p)-2}{p}=& \sum_{\substack{  p \leq  x^{1/2}   }} \frac{\lambda^2_f(p)-2}{p}-\sum_{\substack{  p \leq  x^{1/2} \\ p | d  }} \frac{\lambda^2_f(p)-2}{p} \\
 =& -\log \log x-\sum_{\substack{  p | d  }} \frac{\lambda^2_f(p)-2}{p}+\sum_{\substack{  p >  x^{1/2} \\ p | d  }} \frac{\lambda^2_f(p)-2}{p}
 +O(1),
\end{split}
\end{align}
 where the last estimation above follows from \eqref{merten} and \eqref{M1} in Lemma \ref{RS}. \newline

  Lastly, we note that
\begin{align}
\label{errorpd}
\begin{split}
 \sum_{\substack{  p > x^{1/2} \\ p|d }} \frac{\lambda^2_f(p)-2}{p} \ll \frac 1{x^{1/2}} \sum_{\substack{  p > x^{1/2} \\ p|d }}1
 \ll \frac {\log X}{x^{1/2}\log x}.
\end{split}
\end{align}

 Applying \eqref{sumoverpeval}--\eqref{errorpd} in \eqref{logLupperbound1} this completes the proof of the lemma.
\end{proof}

\section{Proof of Theorem \ref{thm:nonvanish}}
\label{sec:upper bd}

\subsection{Setup}

   We shall follow the approach of A. J. Harper \cite{Harper} and define for a large number $M$,
$$ \alpha_{0} = \frac{\log 2}{\log X}, \;\;\;\;\; \alpha_{j} = \frac{20^{j-1}}{(\log\log X)^{2}} \;\;\; \mbox{for all} \; j \geq 1, \quad \mbox{and} \quad
\mathcal{J} = \mathcal{J}_{X} = 1 + \max\{j : \alpha_{j} \leq 10^{-M} \} . $$

   It follows from the above notations and Lemma \ref{RS} that we have for $X$ large enough,
\begin{align}
\label{sump1}
 \mathcal{J} \leq \log\log\log X , \quad \alpha_{1} = \frac{1}{(\log\log X)^{2}} , \quad \mbox{and} \quad \sum_{p \leq X^{1/(\log\log X)^{2}}}
\frac{1}{p} \leq \log\log X.
\end{align}

  Also, for $1 \leq j \leq \mathcal{J}-1$ and $X$ large enough, we have
\begin{align}
\label{sumpj}
\mathcal{J}-j \leq \frac{\log(1/\alpha_{j})}{\log 20} , \quad \mbox{and} \quad \sum_{X^{\alpha_{j}} < p \leq X^{\alpha_{j+1}}} \frac{1}{p}
 = \log \alpha_{j+1} - \log \alpha_{j} + o(1) = \log 20 + o(1) \leq 10.
\end{align}

  Combining \eqref{sump1} and \eqref{sumpj} yields
\begin{align*}
   \sum_{X^{\alpha_{j-1}} < p \leq X^{\alpha_{j}}}\frac 1{p} \leq \frac{100}{10^{3M/4}}\alpha^{-3/4}_j, \quad 1\leq j  \leq \mathcal{J}.
\end{align*}

For any real numbers $x, y $ with $y \geq 0$, we denote
\begin{align} \label{E}
  E_{y}(x) = \sum_{j=0}^{2\lceil y \rceil} \frac {x^{j}}{j!}.
\end{align}
  We then define for any real number $\alpha$ and any $1\leq j  \leq \mathcal{J}$,
\begin{align*}
 {\mathcal P}_j(p)=&  \sum_{\substack{ X^{\alpha_{j-1}} < q \leq X^{\alpha_{j}} \\ q \text{ prime} }}  \frac{\chi_{8p} (q)\lambda_f(q)}{\sqrt{q}}, \quad {\mathcal M}_j(p, \alpha) = E_{e^2\alpha^{-3/4}_j} \Big (\alpha {\mathcal P}_j(p) \Big ), \quad  {\mathcal M}(p, \alpha)=  (\log X)^{1/2}\prod^{\mathcal{J}}_{j=1} {\mathcal M}_j(p, \alpha).
\end{align*}

  Note that each ${\mathcal M}_j(p,\alpha)$ is a short Dirichlet polynomial of length at most $X^{2\alpha_{j}\lceil e^2k\alpha^{-3/4}_j \rceil}$. By taking $X$ large enough, we have that
\begin{align*}
 \sum^{\mathcal{J}}_{j=1} 2\alpha_{j}\lceil e^2\alpha^{-3/4}_j \rceil \leq \lceil  4e^2k10^{-M/4}\rceil+1.
\end{align*}
   It follows that ${\mathcal M}(p, \alpha)$ is also a short Dirichlet polynomial of length at most $X^{\lceil  4e^210^{-M/4}\rceil+1}$. \newline

   In the sequal, we shall work exclusively with $\mathcal{M}(p) :=\mathcal{M}(p, -1)$ and we remark here that the expression $\mathcal{M}(p)$ can now be regarded as a mollifier approximating $|L(\tfrac{1}{2},f \otimes \chi_{8p})|^{-1}$, similar to those constructed in \cite{LR21} and \cite{DFL21}.

  Now, it is easy to see that in order to prove Theorem \ref{thm:nonvanish}, it suffices to show that, for the function $\Phi$ defined in Section \ref{smoothsum}, we have
\begin{align}
\label{LMM} \sum_{2<p \leq X}(\log p)L(\tfrac{1}{2},f \otimes \chi_{8p}) \mathcal{M}(p)\Phi\leg{p}{X}  \gg & X,  \quad \mbox{and} \\
\label{LM} \sum_{2<p \leq X}(\log p)\left| L(\tfrac{1}{2},f \otimes \chi_{8p}) \right|^2 |\mathcal{M}(p)|^2  \ll & X.
\end{align}

    In fact, \eqref{LMM} easily implies that
\begin{align}
\label{LMM1}
  \sum_{2<p \leq X}(\log p)L(\tfrac{1}{2},f \otimes \chi_{8p}) \mathcal{M}(p)  \gg & X.
\end{align}

  Thus, if we write $\mathcal{N}$ the left-hand side expression in \eqref{Nest}, then \eqref{LM} and \eqref{LMM1} together with the Cauchy-Schwarz inequality imply that
\begin{align*}
\begin{split}
  X^2 \ll &\left( \sum_{2<p \leq X}(\log p)L(\tfrac{1}{2},f \otimes \chi_{8p}) \mathcal{M}(p)\right)^2
  \leq \ \mathcal{N} \sum_{2<p \leq X}(\log p)\left| L(\tfrac{1}{2},f \otimes \chi_{8p}) \right|^2 |\mathcal{M}(p)|^2 \ll \mathcal{N} X.
\end{split}
\end{align*}
 The assertion of Theorem~\ref{thm:nonvanish} then readily follows from the above estimate together with the observation from \eqref{PIT} that as $X \rightarrow \infty$,
\begin{align*}
\begin{split}
  \sum_{\substack{2< p \leq X  }}\log p \sim   X.
\end{split}
\end{align*}

  It now remains to prove  \eqref{LMM} and  \eqref{LM}.  As the proof of \eqref{LMM} is similar to that of \cite[Proposition 2.1]{G&Zhao8} upon using Proposition~\ref{twistedfirstmoment}, we devote the remainder of the paper to the proof of \eqref{LM}.  To that end, we first note the following special case of \cite[Theorem 6.1]{S&Y}.

\begin{prop}
\label{prop: upperbound}
Assume the truth of GRH for $\zeta(s)$, $L(s, f \otimes \chi_{8d})$ as well as the symmetric square $L$-function $L(s, \operatorname{sym}^2 f)$. For any real number $k \geq 0$ and any $\varepsilon>0$, we have
\begin{align*}
    \sumstar_{\substack{0<d<X \\(d,2)=1}} \left| L(\tfrac{1}{2}, f \otimes \chi_{8d})  \right|^{2k}  \ll_k  X(\log X)^{2k^2-k+\varepsilon},
\end{align*}
  where $ \sum^*$ denotes the sum over square-free integers.
\end{prop}

  Now, let
\begin{align*}
{\mathcal M}_{i,j}(p) =& \sum_{\substack{X^{\alpha_{i-1}} < q \leq X^{\alpha_{i}} \\ q \text{ prime} }}  \frac{\chi_{8p} (q)}{q^{1/2+1/(\log X^{\alpha_{j}})}}
\frac{\log (X^{\alpha_{j}}/q)}{\log X^{\alpha_{j}}}, \quad 1\leq i \leq j \leq \mathcal{J} .
\end{align*}
Moreover, we define for $0 \leq j \leq \mathcal{J}-1$,
\begin{align*}
 \mathcal{S}(j) =& \{ 2<p \leq X: | {\mathcal M}_{i,l}(p)| \leq \alpha_{i}^{-3/4} \; \forall \ 1 \leq i \leq j, \; \forall \ i
\leq l \leq \mathcal{J}, \text{but }  | {\mathcal M}_{j+1,l}(p)| > \alpha_{j+1}^{-3/4} \; \text{ for some } j+1 \leq l \leq \mathcal{J} \} ,
\\
 \mathcal{S}(\mathcal{J}) =& \{ 2<p \leq X: |{\mathcal M}_{i,
\mathcal{J}}(p)| \leq \alpha_{i}^{-3/4} \; \forall 1 \leq i \leq \mathcal{J}\}.
\end{align*}

   Note that
$$ \{ 2< p \leq X \}= \bigcup_{j=0}^{ \mathcal{J}} \mathcal{S}(j).  $$
  Thus, in order to establish \eqref{LM}, it suffices to show that
\begin{align*}
   \sum_{j=0}^{\mathcal{J}}\sum_{p \in \mathcal{S}(j)}(\log p) \left| L(\tfrac{1}{2},f \otimes \chi_{8p})  \right|^2|\mathcal{M}(p)|^2
   \ll X .
\end{align*}

  Observe that
\begin{align*}
\begin{split}
\text{meas}(\mathcal{S}(0)) \ll & \sum_{2<p \leq X}
\sum^{\mathcal{J}}_{l=1}
\Big ( \alpha^{3/4}_{1}{|\mathcal
M}_{1, l}(p)| \Big)^{2\lceil 1/(10\alpha_{1})\rceil } \Phi \left( \frac {p}X \right).
\end{split}
\end{align*}

    We use Lemma \ref{lemma logd} to evaluate the last expression above in a manner similar to the treatments given in the proof of \cite[Proposition 2.2]{G&Zhao8} to arrive at
\begin{align}
\label{S0bound}
\text{meas}(\mathcal{S}(0)) \ll &
\mathcal{J}X e^{-1/\alpha_{1}}\ll X e^{-(\log\log X)^{2}/10}  .
\end{align}

   We then deduce via H\"older's inequality that
\begin{align}
\label{LS0bound}
\begin{split}
\sum_{p \in  \mathcal{S}(0)} & (\log p) \left| L(\tfrac{1}{2},f \otimes \chi_{8p})\right|^2 |\mathcal{M}(p)|^2
\leq   \Big ( (\log X)^4 \text{meas}(\mathcal{S}(0)) \Big )^{1/4} \Big (
\sum_{\substack{2<p \leq X}}\left| L(\half, \chi_{8p})^{8} \right |\Big )^{1/4} \Big ( \sum_{\substack{2<p \leq X}} \left|\mathcal{M}(p)\right |^{4} \Big )^{1/2}.
\end{split}
\end{align}

  Similar to the proof of \cite[Proposition 2.2]{G&Zhao8}, we apply Lemma \ref{lemma logd} again to see  that
\begin{align}
\label{N2k2bound}
 \sum_{\substack{2<p \leq X}}\left |\mathcal{M}(p)\right |^{4} \ll X ( \log X  )^{O(1)}.
\end{align}

Also, setting $k=4$ and $\varepsilon=1$ in Proposition \ref{prop: upperbound} implies that
\begin{align}
\label{L8bound}
\sum_{\substack{2<p \leq X}}\left| L(\half, f \otimes \chi_{8p})\right |^{8} \leq \sumstar_{\substack{0<d<X \\(d,2)=1}} \left| L(\tfrac{1}{2}, f \otimes \chi_{8d})  \right|^{8}  \ll
X ( \log X  )^{O(1)}
\end{align}

  We apply the estimations given in \eqref{S0bound}, \eqref{N2k2bound} and \eqref{L8bound} in \eqref{LS0bound} to deduce that
\begin{align*}
\sum_{p \in  \mathcal{S}(0)} & (\log p)\left| L(\tfrac{1}{2},f \otimes \chi_{8p})\right|^2 \left|\mathcal{M}(p) \right |^2   \ll X.
\end{align*}

  Thus it remains to deal with the case in which $j \geq 1$.  If $p \in \mathcal{S}(j)$, we set $x=X^{\alpha_j}$ in \eqref{logLupperbound} and
\begin{align*}
\begin{split}
 & |L(\half, f \otimes \chi_{8p})|^{2} \ll (\log X)^{-1} \exp \Big( \frac {4}{\alpha_j} \Big ) \exp \Big (2
\sum^j_{i=0}{\mathcal M}_{i,j}(p)\Big ).
\end{split}
 \end{align*}

   As  $M_{i, j}(p) \leq  \alpha^{-3/4}_i$ if $p \in \mathcal{S}(j)$, we can apply \cite[Lemma 5.2]{Kirila} to see that
\begin{align*}
\begin{split}
\exp\Big (2
\sum^j_{i=0}{\mathcal M}_{i,j}(p)\Big ) \ll
\prod^j_{i=1}\Big | E_{e^2\alpha^{-3/4}_i}({\mathcal M}_{i,j}(p)) \Big |^2,
\end{split}
 \end{align*}
  where $E_{e^2\alpha^{-3/4}_i}$ is defined as in \eqref{E}. \newline

    We then deduce from the description on $\mathcal{S}(j)$ that when $j \geq 1$,
\begin{align*}
\begin{split}
 \sum_{p \in \mathcal{S}(j)} & (\log p) \left| L(\tfrac{1}{2},f \otimes \chi_{8p})\right|^2 \left|\mathcal{M}(p) \right |^2   \\
 \ll &  (\log X)^{-1} \exp \big(\frac {4}{\alpha_j} \big )
 \sum^{ \mathcal{I}}_{l=j+1} \sum_{\substack{2<p \leq X}} (\log p)\exp \Big ( 2 \sum^j_{i=1}{\mathcal M}_{i,j}(p)\Big )\left|\mathcal{M}(p) \right |^2 \Big ( \alpha^{3/4}_{j+1}{\mathcal
M}_{j+1, l}(p)\Big)^{2\lceil 1/(10\alpha_{j+1})\rceil } \\
\ll &  \exp \big(\frac {4}{\alpha_j}\big )  \sum^{ \mathcal{I}}_{l=j+1}
\sum_{\substack{2<p \leq X}}(\log p)
\prod^j_{i=1}\Big | E_{e^2\alpha^{-3/4}_i}({\mathcal M}_{i,j}(p))E_{e^2\alpha^{-3/4}_i}(-{\mathcal P}_{i}(p))\Big |^2 \\
& \hspace*{1cm} \times \Big | E_{e^2\alpha^{-3/4}_{j+1}}(-{\mathcal P}_{j+1}(p)) \Big |^2\Big ( \alpha^{3/4}_{j+1}{\mathcal
M}_{j+1, l}(p)\Big)^{2\lceil 1/(10\alpha_{j+1})\rceil } \prod^{\mathcal{I}}_{i=j+2} \Big | E_{e^2\alpha^{-3/4}_i}((2k-2){\mathcal P}_{i}(p))\Big |^2.
\end{split}
\end{align*}

We proceed as in the proofs of \cite[Proposition 2.2]{G&Zhao8}, making use of Lemma \ref{lemma logd} and \eqref{sumpj} to arrive at (by noting that $20/\alpha_{j+1}=1/\alpha_j$)
\begin{align*}
\begin{split}
  \sum_{p \in \mathcal{S}(j)} (\log p) \left| L(\tfrac{1}{2},f \otimes \chi_{8p})\right|^2 \left|\mathcal{M}(p) \right |^2
 \ll  X \exp \Big(\frac {4}{\alpha_j} \Big) (\mathcal{I}-j)e^{-122/\alpha_{j+1}}\ll  (\mathcal{I}-j)e^{-42/\alpha_{j+1}}X \ll e^{-1/(10\alpha_{j})}X.
\end{split}
\end{align*}

   As the sum of the right side expression above over $j$ converges, we see that the above estimation implies \eqref{LM}
and this completes the proof of Theorem ~\ref{thm:nonvanish}.

\vspace*{.5cm}

\noindent{\bf Acknowledgments.}   P. G. is supported in part by NSFC grant 11871082 and L. Z. by the Faculty Silverstar Grant PS65447 at the University of New South Wales (UNSW).

\bibliography{biblio}
\bibliographystyle{amsxport}

\vspace*{.5cm}

\noindent\begin{tabular}{p{6cm}p{6cm}p{6cm}}
School of Mathematical Sciences & School of Mathematics and Statistics \\
Beihang University & University of New South Wales \\
Beijing 100191 China & Sydney NSW 2052 Australia \\
Email: {\tt penggao@buaa.edu.cn} & Email: {\tt l.zhao@unsw.edu.au} \\
\end{tabular}

\end{document}